\documentclass[english]{article}
\usepackage[T1]{fontenc}
\usepackage[latin9]{inputenc}
\usepackage{amsmath}
\usepackage{amsthm}
\usepackage{amssymb}
\usepackage[all]{xy}

\makeatletter

\providecommand{\tabularnewline}{\\}

 \theoremstyle{definition}
 \newtheorem*{defn*}{\protect\definitionname}
  \theoremstyle{remark}
  \newtheorem*{rem*}{\protect\remarkname}
  \theoremstyle{definition}
  \newtheorem*{example*}{\protect\examplename}
  \theoremstyle{plain}
  \newtheorem*{prop*}{\protect\propositionname}
  \theoremstyle{plain}
  \newtheorem*{thm*}{\protect\theoremname}
  \theoremstyle{plain}
  \newtheorem*{cor*}{\protect\corollaryname}

\makeatother

\usepackage{babel}
  \providecommand{\corollaryname}{Corollary}
  \providecommand{\definitionname}{Definition}
  \providecommand{\examplename}{Example}
  \providecommand{\propositionname}{Proposition}
  \providecommand{\remarkname}{Remark}
  \providecommand{\theoremname}{Theorem}

\begin{document}

\title{\textbf{\Huge{}Distributive Minimization Comprehensions and the Polynomial
Hierarchy}}

\author{Joaquín Díaz Boils$^{1,}$$^{\ast}$ \\ \small{$^{1}$Facultad de Ciencias Exactas y Naturales.}\\ \small{Pontificia Universidad Cat\'olica del Ecuador.} \\ \small{170150. Quito. Ecuador.}\\ \small{\texttt{$^{\ast}$boils@uji.es}} }
\maketitle
\begin{abstract}
A categorical point of view about minimization in subrecursive classes
is presented by extending the concept of Symmetric Monoidal Comprehension
to that of Distributive Minimization Comprehension. This is achieved
by endowing the former with coproducts and a finality condition for
coalgebras over the endofunctor sending $X$ to $1\oplus X$ to perform
a safe minimization operator. By relying on the characterization given
by Bellantoni, a tiered structure is presented from which one can
obtain the levels of the Polytime Hierarchy as those classes of partial
functions obtained after a certain number of minimizations.

$ $

\textbf{\emph{\normalsize{}Keywords:}}\emph{\normalsize{} }Safe Recursion,
Safe Minimization, Distributive Monoidal Categories, Polytime Hierarchy.
\end{abstract}

\section{Introduction}

The safe interpretation of recursion was introduced by Bellantoni
and Cook in \cite{key-2} and can be used to substitute the bounding
condition in the bounded recursion scheme
\[
\begin{cases}
 & f(u,0)=g(u)\\
 & f(u,x+1)=h(u,x,f(u,x))\\
 & f(u,x)\leq j(u,x)
\end{cases}
\]
under which the subrecursive function classes are closed by a syntactical
condition. The central idea of Bellantoni and Cook was to define two
different kinds of variables (\emph{normal} and \emph{safe variables})
according to the use we make of them in the process of computation.
In \cite{key-2} the class of polynomial time functions has been characterized
in safety terms. In particular, the authors define a class of functions
in the form $f(\overline{x};\overline{y})$ where each input in $f$
is called \emph{normal} or \emph{safe input}, normal inputs are in
the left and separate them from safe by making use of a semicolon. 

In its turn, the \emph{ramified recursion} is a way to avoid \emph{impredicativity}
problems. In a ramified system the objects are defined using levels
such that the definition of an object in level $i$ depends only on
levels below $i$. We will make use of some sets $\mathbb{N}_{k}$,
the \emph{levels of the natural numbers, }since they have a close
relation with different function classes according to their complexity
degree. 

Following the previous ideas, Bellantoni gives in \cite{key-1} a
characterization in safety terms of the known as \emph{Polytime Hierarchy}
as that collection of classes $\square_{i+1}^{P}$:
\begin{itemize}
\item containing the initial functions
\begin{itemize}
\item zero function
\item projections: $\pi_{j}^{m,p}(x_{1},...,x_{m};x_{m+1},...,x_{m+p})=x_{j}$
for $1\leq j\leq m+p$
\item binary successors: $s^{1}(;m)=2m$ and $s^{2}(;m)=2m+1$
\item predecessor: $p(;s^{1}(;0))=p(;s^{2}(;0))=0$ 
\item conditional modulo: $C(;a,b,c)=\begin{cases}
b & if\;a\:mod\:2=0\\
c & otherwise
\end{cases}$ 
\end{itemize}
\item closed under 
\begin{itemize}
\item safe composition: 
\[
f(x;a)=h(\overline{r}(\overline{x};);t(\overline{x};\overline{a}))
\]
\item for $n=1,2$ predicative recursion on notation:
\[
\begin{cases}
f(0,\overline{x};\overline{a})=g(\overline{x};\overline{a})\\
f(s^{n}(y),\overline{x};\overline{a})=h_{n}(y,\overline{x};\overline{a},f(y,\overline{x};\overline{a})) & for\;s^{n}(;0)\neq0
\end{cases}
\]
\end{itemize}
\item and obtained after $i$ applications of safe minimization:\footnote{The safe minimization operator is total and does not entail a notion
of partiality, this is explained in \cite{key-1}.}
\[
f(\overline{x};\overline{a})=\begin{cases}
s^{2}(\mu b.h(\overline{x};\overline{a},b)mod2=0) & if\;there\;is\;such\;b\\
0 & otherwise
\end{cases}
\]
\end{itemize}
In \cite{key-4} the author develops a categorical setting to characterize
subrecursive hierarchies in categorical terms based on the safe and
ramified interpretation of recursion referred to above. For that it
is introduced the concept of \emph{Symmetric Monoidal Comprehension}.
From that construction it is proved that one can perform functions
in a growing classification such as the \emph{Grzegorzcyk} \emph{Hierarchy.} 

The aim of this paper is to extend the results of \cite{key-4} giving
a categorical setting to study safe minimization in the context of
subrecursive classes closed under the operations of safe recursion
and composition, and it is based on former works where that operator
was not considered. In particular, the novelty of this work is to
consider two subindices rather than the one considered in \cite{key-4}.
The first subindex, ranging in $\{0,1\}$, relates to the normal and
safe variables into the functions defined by safe recursion, the second
one, belonging to the set $\{0,...,i-1\}$, is devoted to count the
number of safe minimizations computed to obtain a function in a certain
class of the Polynomial Hierarchy. 

It is known, at the same time, that initiality for algebras associated
to endofunctors $F(-)=1\oplus-$ is used to perform recursion while
finality for the same endofunctor allows to interpret minimization.
Some conditions, essentially sums, have to be added to a monoidal
category to perform that operator, the closing operation required
to get the class of partial recursive functions. In this paper Symmetric
Monoidal Comprehensions are endowed with more structure to obtain,
by means of the safe recursion scheme, a condition of distributivity
and, with the above-mentioned finality condition, safe minimization.
This gives rise to the concept of \emph{Distributive Minimization
Comprehension}, the setting from which we represent partial subrecursive
functions and the Polytime Hierarchy in particular. 

Certain endofunctors $M_{p}$ for $p\in\{0,...,i-1\}$ allow to bound
the number of times that we compute safe minimization in the finality
diagram. After all this is well established, the \emph{Freyd Cover
}plays the role of representability in the context of partial functions
(see \cite{key-6}).

The article is structured as follows: section 2 introduces the basic
concepts and that of Distributive Minimization Comprehension in particular
by giving our main example, in section 3 it is proved that distributivity
is a condition obtained from a $SRR$ scheme while section 4 deals
with safe minimization following the ideas introduced in \cite{key-7},
essentially finality for coalgebras over a certain functor. In section
5 we explain how to represent recursive functions in the free Distributive
Minimization Comprehension and which are the important features satisfied
by it. Finally, in section 6 some conclusions and lines for further
development are given.

\section{Basic structures}

We begin by considering the categories $\Delta^{op}(\mathbf{i},\mathbf{i})$
and $\Delta^{op}(\mathbf{2},\mathbf{2})$, where $\Delta$ is the
simplicial category, as the monoids of endofunctors in $\mathbf{i}$
and $\mathbf{2}$. That is, the categories with objects the natural
numbers lower than $i$ and $2$ and arrows $0\rightarrow1\rightarrow\cdots\rightarrow i-1$
and $0\rightarrow1$ respectively. 
\begin{defn*}
Let be:
\begin{itemize}
\item the functors $T$ and $G$ in $\Delta^{op}(\mathbf{2},\mathbf{2})$
such that $Tk=1$ and $Gk=0$ for $k=0,1$, 
\item for every $p,m\in\mathbf{i}$ the functors $M_{p}$ in $\Delta^{op}(\mathbf{i},\mathbf{i})$
such that 
\[
M_{p}(m)=\left\{ \begin{array}{ll}
p+1 & \text{if }m=p\\
m & \text{if }m\neq p
\end{array}\right.
\]
\item for all $k\in\mathbf{2}$ and $\epsilon:G\Longrightarrow id$ and
$\eta:id\Longrightarrow T$ natural transformations such that
\[
\epsilon(k)=\left\{ \begin{array}{ll}
id_{1} & \text{if }k\neq1\\
0\rightarrow1 & \text{if }k=1
\end{array}\right.\qquad\eta(k)=\left\{ \begin{array}{ll}
id_{0} & \text{if }k\neq0\\
0\rightarrow1 & \text{if }k=0
\end{array}\right.
\]
\end{itemize}
\end{defn*}
A category have the same certain bicategorical property than another
category if the same commutative diagrams are satisfied for both of
them, that is, if there exists a bifunctor between them. For the definition
of \emph{Distributive i-Minimization Comprehension} we consider certain
properties that one category \emph{inherits} from other. This is the
basic categorical structure from which we will develop recursion for
subrecursive (partial) function classes. We endow a categorical structure
with initial diagrams and recursive operators.

In the following the indices range as indicated here: $p,q\in\mathbf{i}$,
$k\in\mathbf{2}$, $n=1,2$ and $\alpha\in\mathbb{N}$. 
\begin{defn*}
A \emph{Distributive i-Minimization Comprehension}, denoted in the
sequel by $(\mathcal{C},T^{\mathcal{C}},G^{\mathcal{C}},\eta^{\mathcal{C}},\epsilon^{\mathcal{C}},M_{p}^{\mathcal{C}})$ 
\begin{enumerate}
\item consists of:
\begin{itemize}
\item a symmetric monoidal category with coproducts $\mathcal{C}$,\footnote{We denote the elements of that structure by $\oplus,in_{r},in_{l},\otimes,\top,l$
and express the objects \emph{modulo associativity }and \emph{symmetry}
in the sequel.}
\item the\emph{ }functors $T^{\mathcal{C}},G^{\mathcal{C}},M_{p}^{\mathcal{C}}:\mathcal{C}\longrightarrow\mathcal{C}$
preserving $\otimes$ and $\oplus$ on the nose\footnote{Preservation \emph{on the nose} means for us equations such as $T^{\mathcal{C}}(A\otimes B)=T^{\mathcal{C}}A\otimes T^{\mathcal{C}}B$,
$T^{\mathcal{C}}(f\otimes B)=T^{\mathcal{C}}f\otimes T^{\mathcal{C}}B$,
$T^{\mathcal{C}}\top=\top$ etc. and same for $G^{\mathcal{C}}$ and
$M_{p}^{\mathcal{C}}$.},
\item natural transformations $\eta^{\mathcal{C}}:id\Longrightarrow T^{\mathcal{C}}$
and $\epsilon^{\mathcal{C}}:G^{\mathcal{C}}\Longrightarrow id$,
\item bifunctors $\Im_{Rec}:\Delta^{op}(\mathbf{2},\mathbf{2})\rightarrow(\mathcal{C},\mathcal{C})$
and $\Im_{Min}:\Delta^{op}(\mathbf{i},\mathbf{i})\rightarrow(\mathcal{C},\mathcal{C})$
such that\footnote{For both $\Im$ to exist we are looking at $\Delta^{op}(\mathbf{i},\mathbf{i})$
and $\Delta^{op}(\mathbf{2},\mathbf{2})$ as bicategories with a unique
0-cells $\mathbf{i}$ and $\mathbf{2}$ respectively.} 
\[
\begin{array}{c}
\Im_{Rec}(T)=T^{\mathcal{C}}\qquad\Im_{Rec}(G)=G^{\mathcal{C}}\\
\\
\Im_{Rec}(\eta)=\eta^{\mathcal{C}}\qquad\Im_{Rec}(\epsilon)=\epsilon^{\mathcal{C}}\qquad\Im_{Min}(M_{p})=M_{p}^{\mathcal{C}}
\end{array}
\]
\end{itemize}
\item containing an object $N_{0,p}$ and three arrows $0_{0,p}$, $s_{0,p}^{1}$
and $s_{0,p}^{2}$ with\emph{ initial diagrams} 
\[
\top\overset{0_{0,p}}{\longrightarrow}N_{0,p}\overset{s_{0,p}^{1}}{\longrightarrow}N_{0,p}\qquad\top\overset{0_{0,p}}{\longrightarrow}N_{0,p}\overset{s_{0,p}^{2}}{\longrightarrow}N_{0,p}
\]
for binary numbers. We define recursively the objects $N_{1,p}$ by
the rule
\[
N_{1,p}=G^{\mathcal{C}}N_{0,p}
\]
and morphisms $0_{1,p}$, $s_{1,p}^{1}$ and $s_{1,p}^{2}$ defined
by $0_{1,p}=G^{\mathcal{C}}(0_{0,p})$ and $s_{1,p}^{1}=G^{\mathcal{C}}(s_{0,p}^{1})=G^{\mathcal{C}}(s_{0,p}^{2})$
giving initial diagrams for $N_{1,p}$. We also have in $\mathcal{C}$

\[
\begin{array}{c}
T^{\mathcal{C}}N_{0,p}=\top\qquad T^{\mathcal{C}}N_{1,p}=N_{1,p}\end{array}
\]
\[
G^{\mathcal{C}}N_{1,p}=N_{1,p}
\]
As well as 
\[
M_{p}^{\mathcal{C}}N_{k,q}=\begin{cases}
\top & \textrm{if \ensuremath{p=q=0}}\\
N_{k,p-1} & \textrm{if \ensuremath{p=q\neq0}}\\
N_{k,q} & \textrm{otherwise}
\end{cases}
\]

\item closed under\emph{ }
\begin{itemize}
\item \emph{flat recursion}:

\noindent for all morphisms 
\[
g:X\longrightarrow Y\textrm{ and }h:N_{0,p}\otimes X\longrightarrow Y
\]
where $X$ and $Y$ are in the form $N_{0,p}^{\alpha}$ there exist
a unique 
\[
FR(g,h):N_{0,p}\otimes X\longrightarrow Y
\]
in $\mathcal{C}$ such that the following diagrams commute
\[
\xymatrix{\top\otimes X\ar[rr]^{0_{0,p}\otimes X}\ar[drr]_{g\circ l} &  & N_{0,p}\otimes X\ar[d]^{FR(g,h)} &  & N_{0,p}\otimes X\ar[dll]^{h}\ar[ll]_{s_{0,p}^{n}\otimes X}\\
 &  & Y
}
\]
 
\item \emph{safe ramified recursion}:

\noindent for all morphisms 
\[
g:X\longrightarrow Y\textrm{ and }h:Y\longrightarrow Y
\]
where $Y$ belongs to the fiber of $T^{\mathcal{C}}$ over $\top$
there exist a unique 
\[
SRR(g,h):N_{1,p}\otimes X\longrightarrow Y
\]
in $\mathcal{C}$ such that the following diagram commutes 
\[
\xymatrix{\top\otimes X\ar[rr]^{0_{1,p}\otimes X}\ar[d]_{l} &  & N_{1,p}\otimes X\ar[rr]^{s_{1,p}^{n}\otimes X}\ar[d]^{SRR(g,h)} &  & N_{1,p}\otimes X\ar[d]^{SRR(g,h)}\\
X\ar[rr]_{g} &  & Y\ar[rr]_{h} &  & Y
}
\]

\end{itemize}
\item and such that every arrow $0_{k,p}\oplus s_{k,p}^{n}$ is an isomorphism
such that the pair 
\[
(N_{k,p},(0_{k,p}\oplus s_{k,p}^{n})^{-1})
\]
is a \emph{bounded terminal coalgebra} for the endofunctor $1\oplus-$
in $\mathcal{C}$ in the following sense:

\noindent for arrows $h_{1},h_{2}:A\rightarrow1\oplus A$ and $f:A\longrightarrow N_{k,p-1}$
there is a unique $\mu f:A\longrightarrow N_{k,p}$ such that the
following diagram commute
\noindent \begin{center}
\hspace{2em}\xymatrix{A\ar[rrr]^-{h_{n}}\ar[d]_-{\mu f} & & & 1\oplus A\ar[d]^-{1\oplus\mu f}\\N_{k,p}\ar[rrr]_-{(0_{k,p}\oplus s_{k,p}^{n})^{-1}} & & & 1\oplus N_{k,p}}
\par\end{center}

\noindent \begin{flushleft}
where $N_{k,p}$ belongs to the fiber of $M_{i-1}^{\mathcal{C}}...M_{0}^{\mathcal{C}}$
over $\top$.
\par\end{flushleft}

\end{enumerate}

\end{defn*}
Flat recursion schemes are actually coproduct diagrams from which,
by applying $G$, we obtain flat recursion also for $N_{1,p}$, they
give the initial diagrams appropriate properties such as the injectivity
of \emph{successor functions} $s^{n}$. 

Moreover, flat recursion schemes allow to define the predecessor function
$p$ given in the Introduction as $FR(0,id)$ as well as the conditional
modulo function $C$ with the help of the \emph{conditional on test
for zero} function $Z$:
\[
\xymatrix{\top\otimes N_{0,p}\ar[rr]^{0_{1,p}\otimes N_{0,p}}\ar[drr]_{\pi_{0}\pi_{1}\circ l} &  & N_{1,p}\otimes N_{0,p}\ar[d]^{Z} &  & N_{1,p}\otimes N_{0,p}\ar[dll]^{\pi_{1}\pi_{1}}\ar[ll]_{s_{1,p}^{n}\otimes N_{0,p}}\\
 &  & N_{0,p}
}
\]
Then, the composition of $Z$ with $mod\:2\:(a)=a\:mod\:2$ defined
by
\[
\xymatrix{\top\ar[rr]^{0_{1,p}}\ar[drr]_{0_{0,p}\circ l} &  & N_{1,p}\ar[d]^{mod\:2}\ar[rr]^{s_{1,p}^{n}} &  & N_{1,p}\ar[d]^{mod\:2}\\
 &  & N_{0,p}\ar[rr]_{1\dot{-}} &  & N_{0,p}
}
\]
gives the conditional modulo function $C$ where $1\dot{-}$ is the
function $a\mapsto1\dot{-}a$ and $\dot{-}$ is the non-negative substraction.
\begin{rem*}
To define $Z$ we have made use of projections which are not at our
disposal unless we are in the context of a cartesian category. But
this is precisely the case for the \emph{free Distributive i-Minimization
Comprehension} defined in section 5 (see Theorem \ref{cartesian}).
\end{rem*}
Condition 4. gives a \emph{safe minimization operator} as explained
in section 3 for the following example. The bounding condition for
that operator over the codomain of $\mu f$ ensures that we do not
compute safe minimization more than $i$ times.
\begin{example*}
\label{ex}Our example of Distributive i-Minimization Comprehension
consists of defining that structure for a presheaf over the category
of sets and partial functions $Set_{P}$. 

Consider the category $Set_{P}^{2\times i}$. Its objects are squares
formed by chains of sets indexed by $\boldsymbol{2}\times\mathbf{i}$:
\textbf{
\[
\xymatrix{X_{0,0}\ar[r]\ar[d] & X_{0,1}\ar[r]\ar[d] & \cdots\ar[r] & X_{0,(i-1)}\ar[d]\\
X_{1,0}\ar[r] & X_{1,1}\ar[r] & \cdots\ar[r] & X_{1,(i-1)}
}
\]
}and its arrows cubes built out of them.

\end{example*}
\begin{itemize}
\item $Set_{P}^{2\times i}$ is a symmetric monoidal category with coproducts,
\item it has as terminal object chains\textbf{
\[
\xymatrix{1\ar[r]\ar[d] & 1\ar[r]\ar[d] & \cdots\ar[r] & 1\ar[d]\\
1\ar[r] & 1\ar[r] & \cdots\ar[r] & 1
}
\]
}denoted by $1^{2\times i}$ where $1$ is any set with a single object
and
\item for $p\in\mathbf{i}$:

\begin{itemize}
\item $0_{k,p}$ give rise to $p-1$ cubes as in the left and $i-p+1$ cubes
such as the one at right:

{\tiny{}
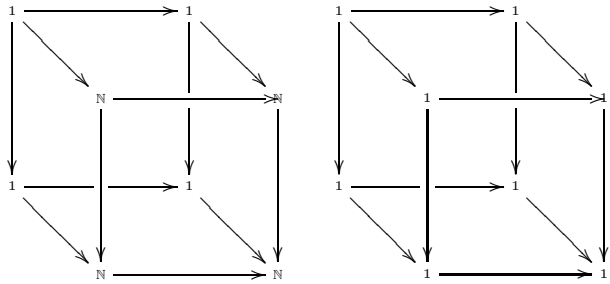
\[
\xymatrix{{1}\ar[rr]\ar[dd]\ar[dr] & \mbox{} & {1}\ar'[d][dd]\ar[dr]\\
\mbox{} & {\mathbb{N}}\ar'[rr]\ar[dd] &  & {\mathbb{N}}\ar[dd]\\
{1}\ar'[r][rr]\ar[dr] &  & {1}\ar[dr]\\
 & {\mathbb{N}}\ar[rr] &  & {\mathbb{N}}
}
\qquad\xymatrix{{1}\ar[rr]\ar[dd]\ar[dr] & \mbox{} & {1}\ar'[d][dd]\ar[dr]\\
\mbox{} & {1}\ar'[rr]\ar[dd] &  & {1}\ar[dd]\\
{1}\ar'[r][rr]\ar[dr] &  & {1}\ar[dr]\\
 & {1}\ar[rr] &  & {1}
}
\]
}{\tiny \par}

\noindent filled with zero and identity arrows
\item $s_{k,p}^{n}$ give rise to $p-1$ cubes as in the left and $i-p+1$
cubes such as the one at right:

{\tiny{}
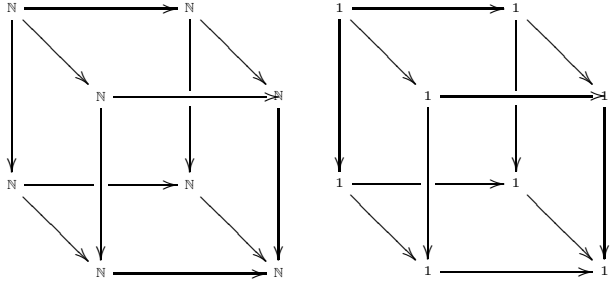
\[
\xymatrix{{\mathbb{N}}\ar[rr]\ar[dd]\ar[dr] & \mbox{} & {\mathbb{N}}\ar'[d][dd]\ar[dr]\\
\mbox{} & {\mathbb{N}}\ar'[rr]\ar[dd] &  & {\mathbb{N}}\ar[dd]\\
{\mathbb{N}}\ar'[r][rr]\ar[dr] &  & {\mathbb{N}}\ar[dr]\\
 & {\mathbb{N}}\ar[rr] &  & {\mathbb{N}}
}
\qquad\xymatrix{{1}\ar[rr]\ar[dd]\ar[dr] & \mbox{} & {1}\ar'[d][dd]\ar[dr]\\
\mbox{} & {1}\ar'[rr]\ar[dd] &  & {1}\ar[dd]\\
{1}\ar'[r][rr]\ar[dr] &  & {1}\ar[dr]\\
 & {1}\ar[rr] &  & {1}
}
\]
}{\tiny \par}

\noindent with binary successors and identity arrows,
\end{itemize}
\item Fixing a single object $X$ there are some special objects in the
form\emph{ }

\textbf{
\[
\xymatrix{X\ar[r]\ar[d] & X\ar[r]\ar[d] & \cdots\ar[r] & 1\ar[d]\ar[r] & 1\ar[d]\\
X\ar[r] & X\ar[r] & \cdots\ar[r] & 1\ar[r] & 1
}
\]
}

\noindent and denoted by $X^{p,q}$ where the chain above is formed
by $p$ objects $X$ and $i-p$ objects $1$ and the chain below is
formed by $q$ objects $X$ and $i-q$ objects $1$. We call these
objects the \emph{levels of} $X$. 
\item We define preserving endofunctors $T^{S}\textrm{ and }G^{S}$ acting
over the columns of $X^{p,q}$ as:

\noindent 
\[
T^{S}\left[\vcenter{\xymatrix{X\ar[d]\\
1
}
}\right]=\vcenter{\xymatrix{1\ar[d]\\
1
}
}\qquad T^{S}\left[\vcenter{\xymatrix{X\ar[d]\\
X
}
}\right]=\vcenter{\xymatrix{X\ar[d]\\
X
}
}
\]
\[
G^{S}\left[\vcenter{\xymatrix{X\ar[d]\\
1
}
}\right]=\vcenter{\xymatrix{X\ar[d]\\
X
}
}\qquad G^{S}\left[\vcenter{\xymatrix{X\ar[d]\\
X
}
}\right]=\vcenter{\xymatrix{X\ar[d]\\
X
}
}
\]
and in general over every arrow
\noindent \begin{center}
$\xymatrix{X_{0,p}\ar[d]\\
X_{1,p}
}
$
\par\end{center}

\noindent as: 

\noindent 
\[
T^{S}\left[\vcenter{\xymatrix{X_{0,p}\ar[d]\\
X_{1,p}
}
}\right]=\vcenter{\xymatrix{1\ar[d]\\
X_{1,p}
}
}\qquad G^{S}\left[\vcenter{\xymatrix{X_{0,p}\ar[d]\\
X_{1,p}
}
}\right]=\vcenter{\xymatrix{X_{1,p}\ar[d]\\
X_{1,p}
}
}
\]
It is obvious that they preserve all tensor and coproducts.
\item While for endofunctors $M_{p}^{S}$ we have for the rows \textbf{
\[
X^{(p)}=\xymatrix{X\ar[r] & \overset{p}{\cdots}\ar[r] & X\ar[r] & 1\ar[r] & \cdots\ar[r] & 1}
\]
}the following table:

\begin{center}
\begin{tabular}{|c|cccccc}
\cline{2-7} 
\multicolumn{1}{c|}{} & \multicolumn{1}{c|}{$X^{(0)}$} & \multicolumn{1}{c|}{$X^{(1)}$} & \multicolumn{1}{c|}{...} & \multicolumn{1}{c|}{$X^{(i-3)}$} & \multicolumn{1}{c|}{$X^{(i-2)}$} & \multicolumn{1}{c|}{$X^{(i-1)}$}\tabularnewline
\hline 
$M_{0}^{S}$ & $1^{2\times i}$ & $X^{(1)}$ & ... & $X^{(i-3)}$ & $X^{(i-2)}$ & $X^{(i-1)}$\tabularnewline
\cline{1-1} 
$M_{1}^{S}$ & $X^{(0)}$ & $X^{(0)}$ & ... & $X^{(i-3)}$ & $X^{(i-2)}$ & $X^{(i-1)}$\tabularnewline
\cline{1-1} 
$\vdots$ & $\vdots$ & $\vdots$ & $\vdots$ & $\vdots$ & $\vdots$ & $\vdots$\tabularnewline
\cline{1-1} 
$M_{i-1}^{S}$ & $X^{(0)}$ & $X^{1}$ & ... & $X^{(i-3)}$ & $X^{(i-3)}$ & $X^{(i-1)}$\tabularnewline
\cline{1-1} 
\end{tabular} 
\par\end{center}

\noindent \begin{flushleft}
while in general $M_{p}^{S}$ acts over every chain
\par\end{flushleft}

\noindent \begin{center}
\hspace{2em}\xymatrix{X_{0,0}\ar[r]^-{h_{0}} & X_{0,1}\ar[r]^-{h_{1}} & \cdots\ar[r]^-{h_{i-2}} & X_{0,(i-1)}}
\par\end{center}

\noindent as: 
\noindent \begin{center}
{\tiny{}\hspace{2em}\xymatrix{X_{0,0}\ar[r]^-{h_{0}} & \cdots\ar[r]^-{h_{i-p}} & X_{0,i-p+1}\ar[r]^-{id} & X_{0,i-p+1}\ar[r]^-{t} & X_{0,i-p+3}\ar[r]^-{h_{i-p+3}} & \cdots\ar[r]^-{h_{i-2}} & X_{0,(i-1)}}}
\par\end{center}{\tiny \par}

\noindent where $t=h_{i-p+2}\circ h_{i-p+1}$, that is, it repeats
the $(i-p+1)-term$. It is obvious that it preserves all tensor and
coproducts. 
\item We define bifunctors $\Im_{Rec}:\Delta^{op}(\mathbf{2},\mathbf{2})\rightarrow(Set_{P}^{2\times i},Set_{P}^{2\times i})$
and $\Im_{Min}:\Delta^{op}(\mathbf{i},\mathbf{i})\rightarrow(Set_{P}^{2\times i},Set_{P}^{2\times i})$
sending $T$, $G$, $\eta$, $\epsilon$ and $M_{p}$ to the respective
endofunctors and natural transformations for $Set_{P}^{2\times i}$.
\item The category of coalgebras for the endofunctor sending an object $X$
to $1\oplus X$ in $Set_{P}^{2\times i}$ is endowed with a number
of isomorphic terminal objects (see section 4).
\end{itemize}

\section{Distributivity}

In this section we prove that, as a consequence of the previous definition,
we are actually endowing $\mathcal{C}$ with a structure of distributive
monoidal category where the distributive arrows $d$ are uniquely
defined by an application of safe ramified recursion: 
\begin{center}
{\scriptsize{}\hspace{2em}\xymatrix{\top\otimes(X\oplus Y)\ar[rr]^-{0\otimes(X\oplus Y)}\ar[d]_{l} &  & N_{1,p}\otimes(X\oplus Y)\ar[rr]^-{s\otimes(X\oplus Y)}\ar[d]^{d_{N_{1,p},X,Y}} &  & N_{1,p}\otimes(X\oplus Y)\ar[d]^{d_{N_{1,p},X,Y}}\\X\oplus Y\ar[rr]_-{(0\otimes X)l_{X}^{-1}\oplus(0\otimes Y)l_{Y}^{-1}} &  & (N_{1,p}\otimes X)\oplus(N_{1,p}\otimes Y)\ar[rr]_-{(s\otimes X)\oplus(s\otimes Y)} &  & (N_{1,p}\otimes X)\oplus(N_{1,p}\otimes Y)}}
\par\end{center}{\scriptsize \par}

The arrows $d$ are actually isomorphisms because of the identities
contained in the following result where the inclusions appearing are
both isomorphisms. We omit subscripts in the sequel.
\begin{prop*}
\label{dist}For every $m$ and arrows $x:\top\rightarrow X,y:\top\rightarrow Y$
the following diagrams commute:
\end{prop*}
\begin{center}
{\scriptsize{}\hspace{2em}\xymatrix{\top\otimes\top\ar[r]^-{\hat{m}\otimes(in_{l}\circ x)}\ar[d]_-{\hat{m}\otimes x} & N_{1,p}\otimes(X\oplus Y)\ar[d]^-{d}\\N_{1,p}\otimes X\ar[r]_-{in_{l}} & (N_{1,p}\otimes X)\oplus(N_{1,p}\otimes Y)}\xymatrix{\top\otimes\top\ar[r]^-{\hat{m}\otimes(in_{r}\circ y)}\ar[d]_-{\hat{m}\otimes y} & N_{1,p}\otimes(X\oplus Y)\ar[d]^-{d}\\N_{1,p}\otimes Y\ar[r]_-{in_{r}} & (N_{1,p}\otimes X)\oplus(N_{1,p}\otimes Y)}}
\par\end{center}{\scriptsize \par}

\noindent where we denote $\hat{m}$ for the arrows $(s_{1,p})^{m}0_{1,p}:\top\rightarrow N_{1,p}$.\footnote{We do not distinguish between $s^{1}$ and $s^{2}$ and write just
$s$ since it does not make any difference.}
\begin{proof}
We proceed by induction over $m$. 
\begin{itemize}
\item For $m=1$ the following diagram is a composition of commuting diagrams:

\hspace{2em}\xymatrix{ & \top\otimes\top\ar[r]^-{\hat{1}\otimes(in_{l}\circ x)}\ar[d]^-{\top\otimes x}\ar@/_{4pc}/[ddd]_-{\hat{1}\otimes x}\ar@{}[dr]|{\mathstrut\raisebox{3ex}{\ensuremath{(\boldsymbol{\alpha})}\kern0.5em }} & N_{1,p}\otimes(X\oplus Y)\ar@/^{2pc}/[ddd]^{d}\ar@{}[dd]|{\mathstrut\raisebox{2ex}{\ensuremath{(\boldsymbol{\delta})}\kern0.5em }}\\ & \top\otimes X\ar[rd]^{l_{X}}\ar[dd]^-{\hat{1}\otimes X}\ar[ur]_-{\hat{1}\otimes in_{l}}\ar@{}[dl]|{\mathstrut\raisebox{3ex}{\ensuremath{(\boldsymbol{\beta})}\kern0.5em }} & \text{}\\\text{} & \text{} & X\ar[d]^-{g}\\ & N_{1,p}\otimes X\ar[r]_-{in_{l}}\ar@{}[ur]|{\mathstrut\raisebox{1ex}{\ensuremath{(\boldsymbol{\gamma})}\kern0.5em }} & (N_{1,p}\otimes X)\oplus(N_{1,p}\otimes Y)}

\noindent where $g=(s0\otimes X)\oplus(s0\otimes Y)\circ(l_{X}^{-1}\oplus l_{Y}^{-1})\circ in_{l}$.
Diagrams $\boldsymbol{\alpha}$, $\boldsymbol{\beta}$ and $\boldsymbol{\gamma}$
commute by direct inspection, $\boldsymbol{\delta}$ commutes because
it can be expressed in the form 

{\scriptsize{}\hspace{.5em}\xymatrix{\top\otimes X\ar[rr]^-{\hat{1}\otimes X}\ar[d]_-{l_{X}}\ar@{}[drr]|{\mathstrut\raisebox{3ex}{\ensuremath{(\boldsymbol{\epsilon})}\kern0.5em }} &  & N_{1,p}\otimes X\ar[rr]^-{s\otimes in_{l}}\ar[d]^-{in_{l}}\ar@{}[drr]|{\mathstrut\raisebox{3ex}{\ensuremath{(\boldsymbol{\eta})}\kern0.5em }} &  & N_{1,p}\otimes(X\oplus Y)\ar[d]^-{d}\\X\ar[rr]_-{g} &  & (N_{1,p}\otimes X)\oplus(N_{1,p}\otimes Y)\ar[rr]_-{f} &  & (N_{1,p}\otimes X)\oplus(N_{1,p}\otimes Y)}}{\scriptsize \par}

\noindent where $f=(s\otimes X)\oplus(s\otimes Y)$. In it $\boldsymbol{\epsilon}$
commutes trivially and $\boldsymbol{\eta}$ commutes because the following
diagram commutes:
\[
\xymatrix{N_{1,p}\otimes X\ar@/^{2pc}/[dr]^{s\otimes in_{l}}\ar@/_{4pc}/[dd]_{in_{l}}\ar[d]^{N_{1,p}\otimes in_{l}}\\
N_{1,p}\otimes(X\oplus Y)\ar[r]^{s\otimes(X\oplus Y)}\ar[d]_{d} & N_{1,p}\otimes(X\oplus Y)\ar[d]^{d}\\
(N_{1,p}\otimes X)\oplus(N_{k+1}\otimes Y)\ar[r]_{f} & (N_{1,p}\otimes X)\oplus(N_{1,p}\otimes Y)
}
\]

\item If we suppose true our result for $\hat{m}$ the following diagrams
commute
\begin{center}
\hspace{2em}\xymatrix{\top\otimes\top\ar[r]^-{(\widehat{m+1})\otimes(in_{l}\circ x)}\ar[d]_-{(\widehat{m+1})\otimes x} & N_{1,p}\otimes(X\oplus Y)\ar[d]^-{d}\\N_{1,p}\otimes X\ar[r]_-{in_{l}} & (N_{1,p}\otimes X)\oplus(N_{1,p}\otimes Y)}
\par\end{center}

\begin{center}
\hspace{2em}\xymatrix{\top\otimes\top\ar[r]^-{(\widehat{m+1})\otimes(in_{r}\circ y)}\ar[d]_-{(\widehat{m+1})\otimes y} & N_{1,p}\otimes(X\oplus Y)\ar[d]^-{d}\\N_{1,p}\otimes Y\ar[r]_-{in_{r}} & (N_{1,p}\otimes X)\oplus(N_{1,p}\otimes Y)}
\par\end{center}

\noindent by composition.
\end{itemize}
\end{proof}
This result, that justifies the word \emph{distributive} in the name
of the structure, can be extended to whatever power of the levels
of natural numbers. That is, the following diagrams also commute
\begin{center}
\hspace{2em}\xymatrix{\top\otimes\top\ar[rr]^-{\left\langle \widehat{m_{1}},...,\widehat{m_{\alpha}}\right\rangle \otimes(in_{l}\circ x)}\ar[d]_-{\left\langle \widehat{m_{1}},...,\widehat{m_{\alpha}}\right\rangle \otimes x} && N_{p,1}^{\alpha}\otimes(X\oplus Y)\ar[d]^-{d}\\N_{1,p}^{\alpha}\otimes X\ar[rr]_-{in_{l}} && (N_{1,p}^{\alpha}\otimes X)\oplus(N_{1,p}^{\alpha}\otimes Y)}
\par\end{center}

\begin{center}
\hspace{2em}\xymatrix{\top\otimes\top\ar[rr]^-{\left\langle \widehat{m_{1}},...,\widehat{m_{\alpha}}\right\rangle \otimes(in_{r}\circ y)}\ar[d]_-{\left\langle \widehat{m_{1}},...,\widehat{m_{\alpha}}\right\rangle \otimes y} && N_{1,p}^{\alpha}\otimes(X\oplus Y)\ar[d]^-{d}\\N_{1,p}^{\alpha}\otimes Y\ar[rr]_-{in_{r}} && (N_{1,p}^{\alpha}\otimes X)\oplus(N_{1,p}^{\alpha}\otimes Y)}
\par\end{center}

\noindent where we consider arrows
\[
\left\langle \widehat{m_{1}},...,\widehat{m_{\alpha}}\right\rangle :\top\rightarrow N_{1,p}^{\alpha}
\]

We have in fact the following relations into a Distributive i-Minimization
Comprehension $(\mathcal{C},T,G,\eta,\epsilon,M_{p})$, ensuring coherence
in the complexity growing structure:

\[
Td_{N_{0,p},X,Y}=d_{\top,TX,TY}\qquad Gd{}_{N_{0,p},X,Y}=d_{N_{1,p},GX,GY}
\]

\[
M_{p}d_{N_{k,p},X,Y}=d_{N_{k,p-1},M_{p}X,M_{p}Y}
\]

\noindent for $k\in\mathbf{2}$, $p\in\mathbf{i}\setminus\{0\}$ and
every $X,Y\in\mathcal{C}$.

\section{Coalgebras and partiality}

In this section we treat partiality in a Distributive i-Minimization
Comprehension. We start from the well known idea that the initial
algebra $(\mathbb{N},1\oplus\mathbb{N}\overset{0\oplus s}{\longrightarrow}\mathbb{N})$
of the endofunctor $1\oplus-$ over the category $Set$ turns out
to be a strong natural numbers object (\emph{nno} in the sequel) where
the uniqueness condition included into it has its counterpart into
the uniqueness of the nno: the equations obtained through a diagram 
\begin{center}
\hspace{2em}\xymatrix{1\oplus{\mathbb{N}}\ar[r]^-{0\oplus{s}}\ar[d]_-{1\oplus{h}} & \mathbb{N}\ar[d]^-{h}\\1\oplus{A}\ar[r] & A}
\par\end{center}

\noindent for every other algebra $(A,1\oplus A\rightarrow A)$ are
equivalent to those obtained in a nno diagram.

It is precisely from duality that we obtain partiality for our recursive
arrows: the terminal coalgebra for $1\oplus-$ over $Set_{P}$ gives
a categorical intuition of minimization (see \cite{key-7}). Let us
denote $F:Set_{P}\rightarrow Set_{P}$ the endofunctor such that $FA=1\oplus A$
and its terminal coalgebra $(\mathbb{N},\mathbb{N}\overset{\alpha}{\longrightarrow}1\oplus\mathbb{N})$
where $\alpha$ turns out to be the isomorphism $(0\oplus s)^{-1}$. 

We spell out the details involved in this construction for the case
of $Set_{P}$. For arrows $h_{1}$, $h_{2}$ and $f:\mathbb{N}^{\alpha_{1}}\otimes\mathbb{N}^{\alpha_{0}}\longrightarrow\mathbb{N}$
the usual coalgebra diagram gives a unique $\mu f:\mathbb{N}^{\alpha_{1}}\otimes\mathbb{N}^{\alpha_{0}}\longrightarrow\mathbb{N}$
such that the following diagrams commute
\begin{center}
\hspace{2em}\xymatrix{\mathbb{N}^{\alpha_{1}}\otimes\mathbb{N}^{\alpha_{0}}\ar[rr]^{h_{n}}\ar[d]_{\mu f} & & 1\oplus(\mathbb{N}^{\alpha_{1}}\otimes\mathbb{N}^{\alpha_{0}})\ar[d]^{1\oplus\mu f}\\\mathbb{N}\ar[rr]_{(0\oplus s^{n})^{-1}} & & 1\oplus\mathbb{N}}
\par\end{center}

\noindent The arrow $\mu f$ is in this way defined by an analogous
of Kleene's minimization from the partial function $f$ (see \cite{key-7}).

For our example \ref{ex} of Distributive i-Minimization Comprehension
$Set_{P}^{2\times i}$ described above we can investigate which is
the form of the objects in the category of coalgebras for the endofunctor
$F$ analogous to the previous one. 

Let $F^{2\times i}:Set_{P}^{2\times i}\longrightarrow Set_{P}^{2\times i}$
be such that
\[
F^{2\times i}\left[\vcenter{\xymatrix{X_{0,0}\ar[r]\ar[d] & X_{0,1}\ar[r]\ar[d] & \cdots\ar[r] & X_{0,(i-1)}\ar[d]\\
X_{1,0}\ar[r] & X_{1,1}\ar[r] & \cdots\ar[r] & X_{1,(i-1)}
}
}\right]=
\]
\[
=\vcenter{\xymatrix{1\oplus X_{0,0}\ar[r]\ar[d] & 1\oplus X_{0,1}\ar[r]\ar[d] & \cdots\ar[r] & 1\oplus X_{0,(i-1)}\ar[d]\\
1\oplus X_{1,0}\ar[r] & 1\oplus X_{1,1}\ar[r] & \cdots\ar[r] & 1\oplus X_{1,(i-1)}
}
}
\]
then the category $CoAlg\:F^{2\times i}$ has:
\begin{itemize}
\item as objects pairs $(X,\alpha)$ where $X$ is of the form 
\[
\xymatrix{X_{0,0}\ar[r]\ar[d] & X_{0,1}\ar[r]\ar[d] & \cdots\ar[r] & X_{0,(i-1)}\ar[d]\\
X_{1,0}\ar[r] & X_{1,1}\ar[r] & \cdots\ar[r] & X_{1,(i-1)}
}
\]
and $\alpha=(\alpha_{0,0},...,\alpha_{0,i-1},\alpha_{1,0},...,\alpha_{1,i-1})$
with $\alpha_{k,p}:X_{k,p}\rightarrow1\oplus X_{k,p}$ for $k\in\mathbf{2}$
and $p\in\mathbf{i}$;
\item as arrows $(X,\alpha)\overset{f}{\longrightarrow}(Y,\beta)$ such
that $\beta\circ f=(F^{2\times i}f)\circ\alpha$.
\end{itemize}
Since $Set_{P}^{2\times i}$ is an indexed category over $Set_{P}$
and $CoAlg\:F$ has as terminal object $(\mathbb{N},(0\oplus s)^{-1})$
the category $CoAlg\:F^{2\times i}$ has isomorphic terminal objects
in the form of pairs formed by 
\begin{itemize}
\item objects in the form
\[
\xymatrix{\mathbb{N}_{0,0}\ar[r]\ar[d] & \cdots\ar[r] & \mathbb{N}_{0,p}\ar[r]\ar[d] & 1\ar[r]\ar[d] & \cdots\ar[r] & 1\ar[d]\\
\mathbb{N}_{1,0}\ar[r] & \cdots\ar[r] & \mathbb{N}_{1,p}\ar[r] & 1\ar[r] & \cdots\ar[r] & 1
}
\]
which by definition of Distributive i-Minimization Comprehension belong
to the fiber of $M_{i-1}^{S}...M_{0}^{S}$ over \textbf{
\[
\xymatrix{1\ar[r]\ar[d] & 1\ar[r]\ar[d] & \cdots\ar[r] & 1\ar[d]\\
1\ar[r] & 1\ar[r] & \cdots\ar[r] & 1
}
\]
}
\item and cubes given by $2i$ arrows in the form $(0_{k,p}\oplus s_{k,p}^{n})^{-1}:\mathbb{N}_{k,p}\rightarrow1\oplus\mathbb{N}_{k,p}$
for $k\in\mathbf{2}$ and $p\in\mathbf{i}$.
\end{itemize}
That is, $F^{2\times i}$ can be endowed with a bounded terminal coalgebra.
We have in this way a safe minimization operator, which is total according
to \cite{key-7}, and is applied i-times as maximum in a Distributive
i-Minimization Comprehension of partial functions.

\section{The free Distributive i-Minimization Comprehension}

By endowing the \emph{initial symmetric monoidal category with coproducts}
with all initial diagrams, the required recursion schemes and the
terminal condition for coalgebras to obtain the minimization operator,
we construct the \emph{free Distributive i-Minimization Comprehension
}which we\emph{ }denote $\mathcal{DM}^{i}$\emph{.} The objects in
$\mathcal{DM}^{i}$ are of the form $\bigoplus(\underset{k\in\mathbf{2},p\in\mathbf{i}}{\bigotimes}N_{k,p})$,
that is, coproducts of finite tensor products of the objects $N_{k,p}$
defined above. Moreover, it can be proved in this case that the tensor
turns out to be a cartesian product. We have in this sense the following:
\begin{thm*}
\label{cartesian}$\mathcal{DM}^{i}$ is a Distributive category.
\end{thm*}
\begin{proof}
See \cite{key-4} together with Proposition \ref{dist}.
\end{proof}
This result allows us to obtain the projection functions belonging
to the Polynomial Hierarchy as defined in the Introduction.

It is precisely $Set_{P}^{2\times i}$ from our example \ref{ex}
that particular Distributive i-Minimization Comprehension in which
we can represent the functions belonging to the i-level of the Polynomial
Hierarchy $\square_{i+1}^{P}$. As in previous studies (\cite{key-4,key-6}
for example) the image of $\mathcal{DM}^{i}$ in $Set_{P}^{2\times i}$
through the Freyd Cover will turn out to be exactly $\square_{i+1}^{P}$.
\begin{defn*}
\label{gamma}The\emph{ standard model of formal morphisms} is the
functor $\Gamma_{i}$ given by the diagram
\[
\xymatrix{\mathcal{DM}^{i}\ar[r]^{\Gamma_{i}} & Set_{P}^{2\times i}}
\]
as an $\mathbf{i}-$indexed version of the \emph{Freyd Cover} for
the functor $\Gamma:\mathcal{DM}^{i}\longrightarrow Set_{P}$ defined
by $\Gamma X=\mathcal{DM}^{i}(\top,X)$ and $\Gamma f=f\circ-$.\footnote{This is a special case of the \emph{global sections functor}.}
\end{defn*}
The syntactical structure here described is connected with the semantics
of numerical functions in the sense that every arrow $\top\rightarrow N_{k,p}$
in $\mathcal{DM}^{i}$ has the form $(s_{k,p}^{n})^{m}0_{k,p}$ for
some $m\in\mathbb{N}$. In connection with this we have the following
result.
\begin{prop*}
\emph{\label{std}}$\mathbb{N}_{k,p}=\{std_{k,p}m/m\in\mathbb{N}\}$
where $std_{k,p}:\mathbb{N}\rightarrow\mathbb{N}_{k,p}$ are defined
by the schemes 
\[
\begin{cases}
std_{k,p}0=0_{k,p}\\
std_{k,p}(s^{n}m)=s_{k,p}(std_{k,p}m)
\end{cases}
\]
with $k\in\mathbf{2}$ and $p\in\mathbf{i}$.
\end{prop*}
\begin{cor*}
$\Gamma N_{k,p}=\mathbb{N}_{k,p}$ for $k\in\mathbf{2}$ and $p\in\mathbf{i}$.
\end{cor*}
This Proposition and its Corollary indicate that the sets generated
by the functor $\Gamma$ applied to the levels of the natural numbers
in $\mathcal{DM}^{i}$ behave as the natural numbers themselves.

Safe composition, as defined in the Introduction, has a representation
in a Distributive Minimization Comprehension by means of diagrams
associated to the natural transformation $\eta$. For an arrow 
\[
f:N_{1,p}^{\alpha}\oplus N_{0,q}^{\beta}\longrightarrow N_{1,r}^{\gamma}
\]
in\emph{ }$\mathcal{DM}^{i}$ we have a commutative diagram in the
form:
\begin{center}
\hspace{2em}\xymatrix{N_{1,p}^{\alpha}\oplus N_{0,q}^{\beta}\ar[r]^-{f}\ar[d]_-{\eta(N_{1,p}^{\alpha}\oplus N_{0,q}^{\beta})} & N_{1,r}^{\gamma}\ar[d]^-{\eta N_{1,r}^{\gamma}}\\T(N_{1,p}^{\alpha}\oplus N_{0,q}^{\beta})\ar[r]_-{Tf} & TN_{1,r}^{\gamma}}
\par\end{center}

This grabs the formulation of safe composition given in the Introduction
because we obtain an expression for morphisms in $\mathcal{DM}^{i}$
with a normal output in terms of other morphisms whose safe inputs
do not have any effect over normal outputs.

\section{Conclusions and future work}

It has been introduced the concept of Distributive i-Minimization
Comprehension to extend the understanding of partiality in subrecursive
functions, that idea has been addressed in \cite{key-3} and in the
context of recursion over arbitrary structures. 

Some of the features of Symmetric Monoidal Comprehensions are inherited
by this new categorical setting, essentially what is related with
the free example. The main novelty of this structure is the double
indexing of the objects relating the two safe operators involved in
the construction of the Polynomial Hierarchy: \emph{safe recursion}
and \emph{safe minimization}. 

The consequence of adding coproducts is a distributive condition which
is satisfied as an application of safe recursive schemes. On the other
hand the terminal diagrams allow to develop minimization and, bounding
the number of these operations that can be computed in every level,
we obtain arrows whose representation in a certain category of sets
are functions in the Polynomial Hierarchy.

There are several lines in which this work could be extended:
\begin{itemize}
\item look for representations of other subrecursive hierarchies of functions
that could be characterized by modifying the concept of Symmetric
Monoidal Comprehension or
\item considering for Distributive i-Minimization Comprehensions, as done
in \cite{key-4} for Symmetric Monoidal Comprehensions, a modal interpretation
of the many-sorted interpretation of recursion introduced primarily
in \cite{key-5}.
\end{itemize}

\end{document}